\documentclass[11pt, oneside]{amsart}
\usepackage[margin=1.5in]{geometry}
\usepackage{cite}
\usepackage[utf8]{inputenc}
\usepackage[T1]{fontenc}
\usepackage[english]{babel}
\usepackage{amsmath,amsfonts,amsthm, mathrsfs}
\usepackage{hyperref}
\usepackage{tikz-cd}

\theoremstyle{plain}
\newtheorem{theorem}{Theorem}[section]

\newtheorem{lemma}[theorem]{Lemma}
\newtheorem{proposition}[theorem]{Proposition}

\theoremstyle{definition}

\newtheorem{remark}[theorem]{Remark}
\numberwithin{equation}{section}
\numberwithin{figure}{section}
\numberwithin{table}{section}

\DeclareMathOperator{\E}{\mathbb E}
\renewcommand{\P}{\operatorname{\mathbb P}}
\DeclareMathOperator{\Var}{Var}
\DeclareMathOperator{\Cov}{Cov}
\DeclareMathOperator{\GL}{GL}

\title{A central limit theorem for cycles of Mallows permutations}

\author{Jimmy He}
\address{Department of Mathematics, MIT, Cambridge, MA  02139}
\email{jimmyhe@mit.edu}
\keywords{Mallows permutations, regenerative processes, cycle structure}
\date{\today}

\begin{document}
\maketitle
\begin{abstract}
Fix $q\neq 1$, and sample $w\in S_n$ from the Mallows measure. We study the distribution of $C_i(w)$, the number of $i$-cycles, as $n$ grows large. When $q<1$, they are jointly Gaussian, and this more or less follows from known ideas, but the regime $q>1$ behaves quite differently. In particular, we show that the even cycles $C_{2i}(w)$ have a mean and variance of order $n$, and jointly converge to Gaussian random variables, while the odd cycles $C_{2i+1}(w)$ have a bounded mean and variance, and converge to $C_{2i+1}(w^{even})$ or $C_{2i+1}(w^{odd})$ for some explicit random permutations $w^{even}$ and $w^{odd}$, depending on whether $n$ is even or odd. An extension to a larger class of functions is also given. The proof utilizes a two-sided stationary regenerative process associated to Mallows permutations constructed by Gnedin and Olshanski \cite{GO12}, extending the ideas of Basu and Bhatnagar \cite{BB17}.
\end{abstract}

\section{Introduction}
This paper studies the cycle structure for random permutations drawn from the Mallows distribution with fixed parameter $q$. Let $C_i(w)$ denote the number of cycles of length $i$ in $w$. For the uniform distribution, Goncharov \cite{G42} showed that the $C_i(w)$ jointly converge to independent Poisson random variables of mean $i^{-1}$. The macroscopic cycles are known to have Poisson-Dirichlet statistics.

There has also been previous work on this problem for Mallows distributed random permutations. Work of Gladkich and Peled \cite{GP18} found the order of the expected length of the cycle containing some given point when $q<1$. They also showed that when $1>q\gg 1-\frac{1}{\sqrt{n}}$, the behaviour of the macroscopic cycles is very similar to that of a uniform permutation, exhibiting Poisson-Dirichlet statistics. There is also work of Mukherjee \cite{M16a}, who studied the regime where $q\approx 1-\frac{\beta}{n}$, and showed that the microscopic cycles converge to independent Poisson random variables, although the means are different from the uniform distribution. In particular, his result applies even if $q>1$, although it is still asymptotically very close to $1$.

Our work complements these results by studying the case when $q$ is fixed. Here, we expect to see behaviour very different from the uniform model, and indeed our results show that the $C_i$ converge to normally distributed random variables with a non-trivial covariance structure. We also study the case when $q>1$, where more interesting behaviour occurs. Mallows permutations with $q>1$ concentrate along the anti-diagonal, and so one expects almost no cycles of odd length. Indeed, we show that the $C_{2i}(w)$ converge to normally distributed random variables with a non-trivial correlation structure, while the $C_{2i+1}(w)$ converge without normalization to limiting random variables depending on whether $n$ is even or odd.

\subsection{Main results}
Fix some permutation $w\in S_n$. An \emph{inversion} of $w$ is a pair $i,j\in \{1,\dotsc, n\}$, with $i< j$, such that $w(i)>w(j)$. The \emph{length} of $w$, denoted $l(w)$, is the number of inversions in $w$. We let $C_i(w)$ denote the number of cycles of size $i$ in the cycle decomposition of $w$, and let $C(w)$ denote the total number of cycles.

The \emph{Mallows distribution} $\mu_q$ is a one-parameter family of probability measures on $S_n$, indexed by a real parameter $q\in (0,\infty)$. A random permutation $w\in S_n$ is Mallows distributed if
	\begin{equation*}
	\P(w=w_0)=\frac{q^{l(w_0)}}{Z_n(q)}
	\end{equation*}
	where $Z_n(q)$ is a normalization constant, given explicitly as
	\begin{equation*}
	Z_n(q)=(1+q)(1+q+q^2)\dotsm (1+q+\dotsc+q^{n-1}).
	\end{equation*}
	When $q=1$, it reduces to the uniform distribution and when $q\to 0$ or $q\to\infty$, it degenerates to a point mass at the identity or the permutation sending $i$ to $n-i+1$ respectively.

Our first result gives the asymptotic means and covariances of the $C_i(w)$, as well as their limiting distribution, when $q<1$.
\begin{theorem}
\label{thm: q<1}
Let $q\in (0,1)$, and let $w\in S_n$ be Mallows distributed with parameter $q$. Let $C_i(w)$ denote the number of $i$-cycles in $w$, and let $C(w)$ denote the total number of cycles. Then there exist constants $\alpha_i,\beta_{ij}>0$, depending only on $q$, such that
\begin{equation*}
    \E(C_i(w))\sim \alpha_i n,\qquad \Cov(C_i(w),C_j(w))\sim \beta_{ij}n,
\end{equation*}
and
\begin{equation*}
    \left(\frac{C_i(w)-\alpha_i n}{\sqrt{n}}\right)_{i\geq 1}\xrightarrow{(d)}(Z_i)_{i\geq 1},
\end{equation*}
where the $Z_i$ are normal random variables of mean $0$ and covariance structure $\beta_{ij}$.

In addition, if $\alpha=\sum \alpha_i$, there exists a constant $\beta>0$ such that
\begin{equation*}
    \E(C(w))\sim \alpha n,\qquad\Var(C(w))\sim \beta n,
\end{equation*}
and
\begin{equation*}
    \frac{C(w)-\alpha n}{\sqrt{\beta n}}\xrightarrow{(d)}Z,
\end{equation*}
where $Z\sim N(0,1)$.
\end{theorem}

The more interesting case is when $q>1$, which is covered by the next theorem.
\begin{theorem}
\label{thm: q>1}
Let $q\in (1,\infty)$, and let $w\in S_n$ be Mallows distributed with parameter $q$. Let $C_i(w)$ denote the number of $i$-cycles in $w$, and let $C(w)$ denote the total number of cycles. Then there exist constants $\alpha'_{i},\beta'_{ij}>0$, depending only on $q$, such that
\begin{equation*}
    E(C_{2i}(w))\sim \alpha'_{i}n,\qquad\Cov(C_{2i}(w),C_{2j}(w))\sim \beta'_{ij}n,
\end{equation*}
and
\begin{equation*}
    \left(\frac{C_{2i}(w)-\alpha'_i n}{\sqrt{n}}\right)_{i\geq 1}\xrightarrow{(d)}(Z_i)_{i\geq 1},
\end{equation*}
where the $Z_i$ are normal random variables of mean $0$ and covariance structure defined by $\beta_{ij}'$.

There exist random permutations $w^{even}$ and $w^{odd}$ (of a random length) such that the joint distribution of the odd cycles $(C_{2i+1}(w))_{i\geq 0}$ converges to $(C_{2i+1}(w^{even}))_{i\geq 0}$ or $(C_{2i+1}(w^{odd}))_{i\geq 0}$, depending on whether $n\to\infty$ is even or odd.

Finally, if $\alpha'=\sum \alpha'_{i}$, there exists a constant $\beta'>0$ such that
\begin{equation*}
    \frac{C(w)-\alpha' n}{\sqrt{\beta' n}}\xrightarrow{(d)}Z,
\end{equation*}
where $Z\sim N(0,1)$.
\end{theorem}

\begin{remark}
	The constants $\alpha_i$, $\beta_{ij}$, $\beta$, $\alpha_i'$, $\beta_{ij}'$, and $\beta'$ do not in general appear to have nice formulas, but can be given implicitly as moments of the number of cycles in a related random permutation, see Remark \ref{rmk: constants}. For $q<1$, a characterization was given in \cite[Proposition 3.3]{PT19}.
	
	However, $\alpha_1$ has the formula
	\begin{equation*}
	    \alpha_1=\frac{1-q}{q}(q;q)_\infty\sum_{j=0}^\infty \frac{q^{(j+1)^2}}{(q;q)_j^2},
	\end{equation*}
	where $(a;q)_r=\prod _{i=0}^{r-1}(1-aq^i)$ denotes the $q$-Pochhammer symbol. This series converges rapidly and allows reasonable asymptotic estimates. This is stated as Proposition \ref{prop: a1}.
\end{remark}

\begin{remark}
\label{rmk: ext}
In fact, Theorems \ref{thm: q<1} and \ref{thm: q>1} hold for a more general class of functions. We make a few more comments on the details of this extension in Remark \ref{rmk: proof of ext}.

Let $w\in S_n$, and suppose that $w=w_1w_2$ where $w_1$ sends $[1,i]$ to itself and $[i+1,n]$ to itself. Say that a function $f:\cup _n S_n\to \mathbf{R}^d$ is \emph{additive} if $f(w)=f(w_1)+f(w_2)$ for all $w=w_1w_2$ decomposing in this way, where $f(w)$ for $w$ a permutation on an interval $[i,j]$ is defined by shifting down the permutation to $[1,j-i+1]$. Then Theorem \ref{thm: q<1} holds for any (non-trivial) additive function satisfying $|f(w)|\leq Cn^k$ for $w\in S_n$, where $C$ and $k$ are constants.

Let $w\in S_n$, and suppose that $w=w_1w_2$, where $w_1$ sends $[i,n-i]$ to itself and $w_2$ sends $[1,i]$ to $[n-i+1,n]$ and vice versa. Say that $f:\cup_n S_n\to\mathbf{R}^d$ is \emph{anti-additive} if $f(w)=f(w_1)+f(w_2)$ for all $w=w_1w_2$ decomposing in this way. Then Theorem \ref{thm: q>1} holds for any (non-trivial) anti-additive function satisfying $|f(w)|\leq Cn^k$ for $w\in S_n$, where $C$ and $k$ are constants.

Here, non-trivial means that $f$, when restricted to permutations of size $n$ for which no non-trivial decomposition of the form $w=w_1w_2$ exists, is non-constant. This assumption is needed to ensure that the variances $\beta_{ii}$ and $\beta_{ii}'$ are non-zero. Almost any reasonable function satisfies this, but note in particular that $f(w)=n$ the size of the permutation does not satisfy this, and indeed Theorems \ref{thm: q<1} and \ref{thm: q>1} obviously fail for this function.

Note that if $f(w)$ is either additive or anti-additive, then so is $f(w^{-1})$. Thus, the theorems also apply to joint statistics for a Mallows permutation and its inverse, giving another proof of Theorem 1.2 of \cite{H21}.
\end{remark}

The proof of Theorem \ref{thm: q<1} follows the ideas of Basu and Bhatnagar \cite{BB17}, who studied the longest increasing subsequence for Mallows permutations. A strong law of large numbers for the number of cycles was established by Pitman and Tang \cite{PT19}, who studied a much more general class of examples. While this result does not appear to have been written down in the literature, it follows from known ideas. We give a proof of this theorem for completeness.

The proof of Theorem \ref{thm: q>1} is more involved and requires some new ideas. In particular, while the Mallows distribution of parameter $q$ can be related to the Mallows distribution of parameter $q^{-1}$ by reversing the permutation, this significantly affects the cycle structure. 

We use a two-sided stationary regenerative process constructed by Gnedin and Olshanski \cite{GO12}, which we call the \emph{stationary Mallows process}. We then study a related regenerative process which we call the \emph{symmetric process}, where the regeneration times are given as the times where both sides of the stationary Mallows process simultaeneously regenerate. This corresponds to a decomposition $w=v_0\dotsm v_k$, where $v_0$ sends some interval $[s_0+1,n-s_0]$ to itself, and $v_i$ sends $[s_i+1,s_{i-1}]$ to $[n-s_{i-1}+1,n-s_{i}]$ and vice versa. Moreover, the different $v_i$ will be independent due to the regenerative structure of the stationary Mallows process. Since the cycles are entirely restricted to some $v_i$, we can thus obtain Theorem \ref{thm: q>1} using a central limit theorem for regenerative processes.

The biggest technical challenge in the proof is to obtain the necessary moment estimates for the regeneration times in the symmetric process. This corresponds to obtaining moment estimates for the simultaneous return times of two independent copies of a Markov chain originally used in \cite{BB17} to study the regeneration times in the Mallows process.

\begin{remark}
	After the completion of this work, the author learned that Theorems \ref{thm: q<1} and \ref{thm: q>1} were also obtained independently by M\"uller and Verstraaten \cite{MV22} using similar methods.
\end{remark}

\subsection{Background}
The works most closely related to ours is the recent development of a theory of regenerative permutations and its application to the study of Mallows permutations. Gnedin and Olshanski \cite{GO10,GO12} constructed infinite extensions of Mallows measure to $\mathbf{N}$ and $\mathbf{Z}$. This was subsequently generalized by Pitman and Tang \cite{PT19}, who also gave a very clear account of the regenerative structure within these infinite permutations. This was in part inspired by work of Basu and Bhatnagar \cite{BB17}, who applied these ideas to study the longest increasing subsequence for Mallows permutations. Gladkich and Peled \cite{GP18} also used related ideas to study the cycle structure. These applications all rely on taking $q<1$.

This work is in some sense a continuation of this body of work. We explain what properties a permutation statistic needs to have to apply these methods when $q<1$. We also extend the method to handle the case when $q>1$, although the functions that this can be applied to are more limited, and we cannot for example say anything about the longest increasing subsequence for Mallows permutations when $q>1$.

The Mallows distribution was introduced by Mallows to study non-uniform ranked data \cite{M57}. It is perhaps the most widely used non-uniform distribution on permutations in applied statistics, see \cite{M16b} or \cite{T19} for discussion of the statistical uses for Mallows permutations. Thus, understanding the behaviour of features of Mallows permutations is an important problem. They have also seen applications to the study of one-dependent processes \cite{HHL20} and stable matchings \cite{AHHL18}.

A growing body of work studies the behaviour of Mallows permutations as the parameter $q$ varies. In general, it seems that for $q$ sufficiently close to $1$, the model behaves very similarly to the uniform permutation, while for $q$ far enough from $1$, different behaviour occurs. This has been observed for the longest increasing subsequence \cite{MS13, BP15, BB17}.

One interesting feature of the Mallows model is that in many cases, there is sharp phase transition between these two regimes. Specifically, this has been observed for the expected length of cycles \cite{GP18}. It would be interesting to see if the fluctuations of the number of cycles also have a phase transition. Our work suggests that this would be a transition from Poisson to Gaussian fluctuations.

Finally, let us mention some other recent work on Mallows permutations. The Mallows distribution appeared as the induced distribution on double cosets coming from the uniform distribution on $\GL_n(\mathbf{F}_q)$ for the Bruhat decomposition \cite{DS21}. Recent work of Pinsky \cite{P21,P20} studies pattern occurences in Mallows permutations as well as the secretary problem when the arrivals are biased to be a Mallows permutation. Finally, in \cite{ABC21}, the height of trees constructed from Mallows permutations was studied.
	
\subsection{Organization}
The paper is organized as follows. In Section \ref{sec: regen process}, we recall some important facts about regenerative and renewal processes. In Section \ref{sec: mallows process}, we introduce the main tool, the Mallows process, which is an infinite regenerative process closely related to Mallows permutations. In Section \ref{sec: moment estimates}, we show that the regeneration times of the symmetric process have finite moments, which allow us to apply tools from renewal theory. In Section \ref{sec: main}, we prove Theorems \ref{thm: q<1} and \ref{thm: q>1}.

\section{Regenerative processes}
\label{sec: regen process}
In this section, we recall some key facts about regenerative processes, including the renewal central limit theorem.

Recall that a (discrete time) regenerative process is a process $X(t):\mathbf{N}\to X$ on $\mathbf{N}$ such that there exist (random) times $T_i$ for $i\geq 1$ for which $\{X(\sum_{i\leq j} T_i+t)\}_{t\geq 0}$ is equal in distribution to $\{X(t)\}_{t\geq 0}$ for any $j$, and independent of the process before time $\sum_{i\leq j} T_i$. The times $\sum _{i\leq j}T_i$ are called the regeneration times, and the process between regeneration times, which we will denote $X_i$, are independent and identically distributed (as random variables taking values in finite sequences in $X$). 

Actually, to handle the stationary Mallows process, we will need to consider a slightly more general definition which allows the first regeneration time to have a different distribution from the remaining ones. These are called \emph{delayed} regenerative processes. In this case, we will let $T_0$ denote the first $T_i$, and $X_0$ denote the first $X_i$, and the subsequent ones we will labeled as usual. Thus, a non-delayed regenerative process simply means that $T_0=0$ and $X_0$ is empty.

We let $N(t)$ be the largest $n$ such that $\sum_{i\leq n}T_i\leq t$. Let $f:\cup X^k\to\mathbf{R}^d$, and consider the random sum $\sum_{i=1}^{N(t)}f(X_i)$. The following result of Smith establishes a central limit theorem for this sum, given some mild moment assumptions (actually the proof assumes $T_0=0$ but it is stated that the result holds without this assumption, see \cite{GW02} for example for the extension in the univariate case).
\begin{theorem}[{\hspace{1sp}\cite[Theorem 10]{S55}}]
\label{thm: regen CLT}
Suppose that $\E(T_i)=\mu<\infty$ and $\E(T_i^2)<\infty$, and that the vector $f(X_i)$ has finite first and second moments for all $i$. Then
\begin{equation*}
    \frac{\sum_{i=1}^{N(t)}f(X_i)-\alpha t}{(t/\mu)^{\frac{1}{2}}}\xrightarrow{(d)} Z,
\end{equation*}
where $Z\sim N(0,\Sigma)$, $\alpha_i=\E(f(X_1)_i)/\mu$ and
\begin{equation*}
    \Sigma_{ij}=\Cov\left(f(X_1)_i-\frac{\E(f(X_1)_i)T_i}{\mu},f(X_1)_j-\frac{\E(f(X_1)_j)T_j}{\mu}\right).
\end{equation*}
\end{theorem}
\begin{remark}
Theorem \ref{thm: regen CLT} was proved in much more generality in the original paper of Smith \cite{S55}, where the result was shown for continuous time. We restrict ourselves to the case when the regeneration times are restricted to $\mathbf{N}$, and the so-called cumulative process is a pure jump process with jumps only at the regeneration times.
\end{remark}
We also require the asymptotic convergence of the moments. This theorem seems well-known, and in particular follows from Theorem \ref{thm: regen CLT} and the fact that the mean and variances for each component converge \cite[Theorem 8]{S55} (again, the proof assumes $T_0=0$ but as long as $T_0$ has finite moments, this assumption is not necessary).
\begin{theorem}[{\hspace{1sp}\cite{S55}}]
\label{thm: moment conv}
With the same setup as Theorem \ref{thm: regen CLT}, we have
\begin{equation*}
    \frac{1}{t}\E\left(\sum_{i=1}^{N(t)}f(X_i)\right)\to\alpha
\end{equation*}
and
\begin{equation*}
    \frac{1}{t}\Cov\left(\sum_{i=1}^{N(t)}f(X_i)\right)\to \frac{1}{\mu}\Sigma
\end{equation*}
as $t\to\infty$.
\end{theorem}

We will also need to ensure that no single $T_i$ can cause any issues. For this, we require the following limit result for the length of the interval containing $n$, which goes to the size-bias distribution, see for example \cite[Eq. 5.70]{G13}.

\begin{proposition}
\label{prop: size-bias}
Let $T^{(n)}$ be the length of the interval containing $n$. Then $T^{(n)}\xrightarrow{(d)} T^*$ as $n\to\infty$, where $T^*$ has the size-bias distribution with respect to the $T_i$. That is, $T^*$ has the unique distribution such that $\E(f(T^*))=\E(T_1f(T_1))$ for all $f$ where the expectations make sense.
\end{proposition}

\section{Mallows process}
\label{sec: mallows process}
In this section, we recall the construction of the Mallows process on both $\mathbf{N}$ and $\mathbf{Z}$, which is thus named because we want to view it as a discrete time stochastic process. The analysis of Mallows permutations when $q<1$ is fairly straightforward using the regenerative properties of this process. The analysis when $q>1$ is more interesting, and requires the use of the process on $\mathbf{Z}$, started at stationarity. We then study a related regenerative process we call the symmetric process, and give some moment estimates needed to apply the renewal theory arguments to obtain Theorem \ref{thm: q>1}.

The Mallows process was defined by Gnedin and Olshanski \cite{GO10}, who later extended the definition to a two-sided process \cite{GO12}. A more general notion of regenerative permutation on $\mathbf{N}$ or $\mathbf{Z}$ was introduced in \cite{PT19}. These ideas were first used by Basu and Bhatnagar \cite{BB17} to study the longest increasing subsequence for Mallows permutations (with related ideas appearing in \cite{GP18}), and then in \cite{H21} to study descents in $w$ and $w^{-1}$. The basic properties of the Mallows process are reviewed here without proof. We give references to places where more details can be found.

\subsection{Mallows process}
\begin{figure}
    \centering
    \includegraphics{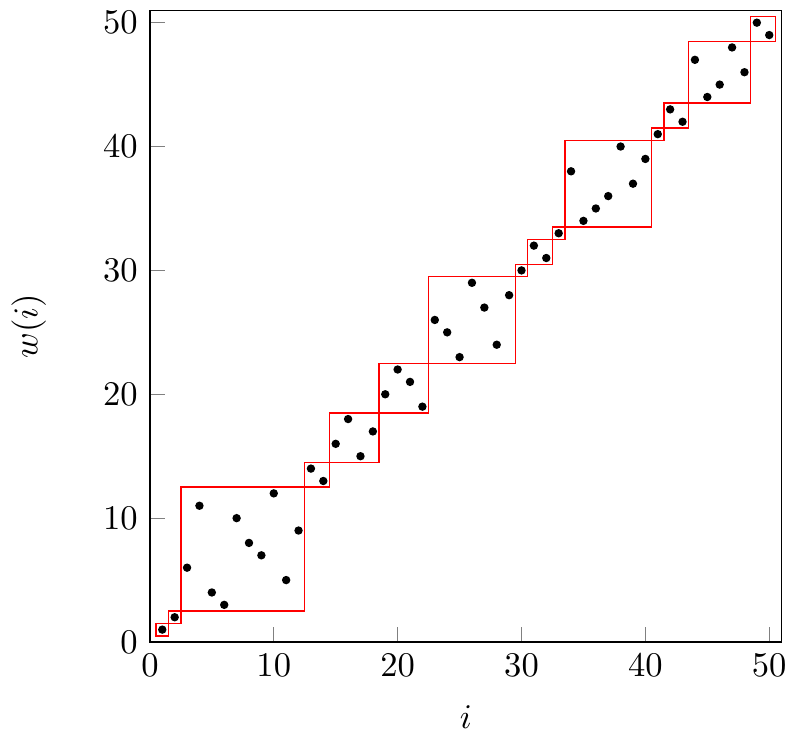}
    \caption{Mallows process $\widehat{w}$ on $\mathbf{N}$ is shown as black dots. Red squares mark excursions. The $T_i$ are the lengths of the red squares, and the $w_i$ are the permutations within each red square.}
    \label{fig: mallows process}
\end{figure}

Fix $0<q<1$. The \emph{Mallows process} is a random permutation $w:\mathbf{N}\to\mathbf{N}$ defined as follows. Set $\widehat{w}(1)=i$ with probability $q^{i-1}(1-q)$ and given $\widehat{w}(1), \dotsc, \widehat{w}(i)$, set $\widehat{w}(i+1)=k\in \mathbf{N}\setminus \{\widehat{w}(1),\dotsc,\widehat{w}(i)\}$ with probability $q^{k'-1}(1-q)$ where $k'-1$ is the number of elements in $\mathbf{N}\setminus \{\widehat{w}(1),\dotsc,\widehat{w}(i)\}$ less than $k$. This is almost surely a bijection.

Given any $n\in\mathbf{N}$, the relative order of $\widehat{w}(1),\dotsc,\widehat{w}(n)$ as a random element in $S_n$ is Mallows distributed, and more generally the same is true for the relative order of $\widehat{w}(i)$ for $i\in [a,b]$. We will denote the permutation in $S_{b-a+1}$ induced by this relative order $w^{a,b}$. Thus, it often suffices to study the Mallows process, which has a regenerative structure that we now describe.

Let $T_1$ be the first time $n$ that $\widehat{w}(i)\leq n$ for all $i\leq n$. In other words, this is the first $n$ for which $w(1),\dotsc,w(n)$ defines a permutation of $1,\dotsc,n$. Then the process $(\widehat{w}(T_1+k)-T_1)_{k\geq 1}$ is equal in distribution to $(\widehat{w}(k))_{k\geq 1}$.

Similarly, let $T_k+\dots+T_1$ be the $k$th time that this occurs. Let $w_1$ be the permutation induced by $\widehat{w}(1),\dotsc,\widehat{w}(T_1)$ and let $w_k$ be the permutation induced by $\widehat{w}(T_{k-1}+1),\dotsc,\widehat{w}(T_k)$ for $k\geq 1$. Then it's clear that the $T_k$ and the $w_k$ are independent and identically distributed. Call the $w_k$ \emph{excursions} in the Mallows process and $T_k$ their \emph{sizes}. Moreover, the times $T_k$ are the renewal times of a renewal process. See Figure \ref{fig: mallows process} for an example of this process up until $t=50$. 

The cycles of $w$ must lie in the irreducible components. Thus we can study $w^{1,n}$, which has the same distribution, and instead count cycles within each $w_i$ contained in $[1,n]$. Standard renewal theory arguments will then establish Theorem \ref{thm: q<1}, as long as the moments of $T_k$ are finite. This is the content of the next proposition.

\begin{proposition}[{\hspace{1sp}\cite[Proposition 3.4]{H21}}]
\label{prop: T_i estimates}
	We have $\E(T_i^k)<\infty$ for all $i>0$.
\end{proposition}

\subsection{Stationary Mallows process}
The stationary Mallows process should be thought of as the Mallows process started from stationarity, in the sense that $\widehat{w}(i)-i$ is a stationary process. It was first constructed in \cite{GO12}, although they did not exploit the regenerative structure. Later, in \cite{PT19}, the fact that this process can be viewed as the stationary version of the Mallows process was explained and the usual renewal-theoretic consequences were given.

For example, stationary Mallows process can be obtained by running the Mallows process for a long time, see \cite[Lemma 3.1]{PT19}. We do not give a formal construction as it is quite a bit more involved than the Mallows process, but merely take for granted the existence of a process $\widehat{w}$ with certain desired properties. See \cite{GO12} and \cite[Section 3]{PT19} for details and proofs.

Let $\widehat{w}$ denote the \emph{stationary Mallows process}, as defined in \cite{GO12}. This is a random bijection $\mathbf{Z}\to\mathbf{Z}$. This process again has a regenerative structure. We let $T_0$ denote the length of $w_0$, the irreducible permutation containing $0$, and we let $T_i$ for $i\in\mathbf{Z}$ denote the lengths of the $i$th irreducible permutation before or after $w_0$, denoted $w_i$. Then the $T_i$ and $w_i$ are independent for different $i$, and for $i\neq 0$ they are identically distributed according to the same distribution as for the Mallows process described previously. We note that $T_0$ and $w_0$ do not have the same distribution as the other $T_i$ and $w_i$. The size of the first excursion $T_0$ has the size-biased distribution with respect to the distribution of the other $T_i$, as is the case for any stationary renewal process. Thus, Proposition \ref{prop: T_i estimates} immediately extends to the stationary Mallows process.

\begin{proposition}
Let the $T_i$ denote the sizes of the excursions in the stationary Mallows process. Then $\E(T_i^k)<\infty$ for all $i\in\mathbf{Z}$.
\end{proposition}

Again, $w^{a,b}\in S_{b-a+1}$ is Mallows distributed of parameter $q$. This can be seen from the fact that this process is $q$-exchangeable (see comments after Theorem 6.1 in \cite{GO12} for example).

Finally, we remark that $-\widehat{w}(-i)$ has the same distribution as $\widehat{w}$. Thus, the stationary Mallows process is also time-reversal invariant.

\subsection{The symmetric process}
\begin{figure}
    \centering
    \includegraphics{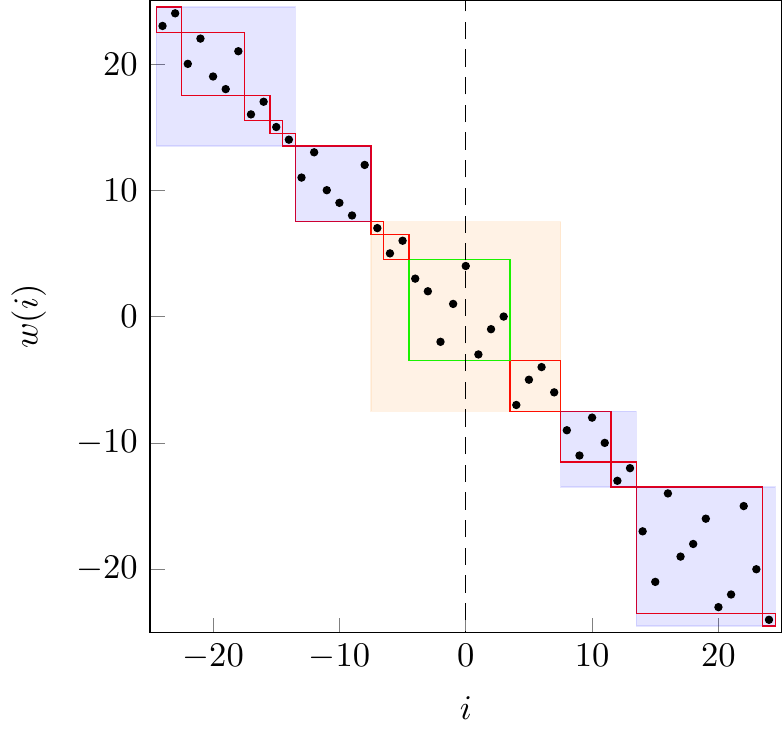}
    \caption{The reversed stationary Mallows process $\widehat{w}^{rev}$ is shown as black dots. Here, we consider the odd case, which is centered at $0$, marked by the dashed line. Red squares mark the excursions in the stationary Mallows process $\widehat{w}$, with the green square denoting the excursion containing $0$. The orange shaded square denotes the first regeneration in the symmetric process, $S_0$ is the length of the orange shaded square, $v_0$ is the permutation contained in the orange shaded square. The blue shaded squares, paired symmetrically, denote the subsequent regenerations in the symmetric process, $S_i$ is the sum of the lengths of a symmetric pair of blue shaded squares, and $v_i$ is the permutation contained in a symmetric pair of blue shaded squares.}
    \label{fig: sym process odd}
\end{figure}

Note that all the processes we have defined are for $q<1$. In this section, we define the symmetric process, which is essentially just a different regenerative structure on $\widehat{w}$, which will be essential in studying the cycle structure when $q>1$.

For a Mallows permutation $w$, the reversal permutation $w^{rev}(i)=w(n-i+1)$ is Mallows distributed with parameter $q^{-1}$. Thus, we would wish to study the analogous Mallows process given by $\widehat{w}^{rev}=\widehat{w}(-i)$. The relative order of $\widehat{w}^{rev}(i)$ for $a+1\leq i\leq b$ is Mallows distributed in $S_{b-a}$ of parameter $q^{-1}$. However, it is no longer true that $w_i$ is irreducible, and so it is no longer possible to study the cycle structure using the same renewal structure. This is what motivates the following construction. A related idea was proposed in \cite{GP18}. 

We will need slightly different constructions depending on whether we wish to study Mallows permutations for $n$ even or $n$ odd. We begin with $n$ odd. We refer the reader to Figure \ref{fig: sym process odd} as a visual guide for the construction.

Let $T_0^+$ and $T_0^-$ be the length of $w_0$ to the right and left of $0$ respectively, so $1+T_0^++T_0^-=T_0$. In Figure \ref{fig: sym process odd}, this is the length of the green square to the left and right of $0$, so $T_0^+=3$ and $T_0^-=4$.

Let $k_0^-$ and $k_0^+$ be the first times that $T_0^++\sum_{i=1}^{k_0^+} T_i=T_0^-+\sum_{i=1}^{k_0^-}T_{-i}$. In Figure \ref{fig: sym process odd}, these are the number of red squares to the left and right of the green square in the orange shaded square, so $k_0^-=2$ and $k_0^+=1$.

Let $v_0$ be the permutation obtained by concatenating $w_i$ for $-k_0^-\leq i\leq k_0^+$, and let $S_0$ denote its length. In Figure \ref{fig: sym process odd}, $v_0$ is the permutation contained in the orange shaded square, and $S_0$ is the length of the orange shaded square, so $S_0=15$. It can be seen that $v_0$ is simply the first time that $\widehat{w}^{rev}$ sends $[-i,i]$ to itself.

Now given $k_i^-$, $k_i^+$, inductively define $k_{i+1}^-$ and $k_{i+1}^+$ as the smallest integers such that $\sum _{i=k_{i}^-+1}^{k^-_{i+1}}T_{-i}=\sum_{i=k_i^++1}^{k_{i+1}^+}T_i$, and let $S_{i+1}=2\sum _{i=k_{i}^-+1}^{k^-_{i+1}}T_{-i}$. In other words, we are waiting for the first time after $S_i$ that $\widehat{w}^{rev}$ sends the interval $[S_{i}/2,S_{i+1}/2]$ to $-[S_{i}/2,S_{i+1}/2]$ and vice versa. In Figure \ref{fig: sym process odd}, note that the blue shaded squares are paired on the left and right, each having the same size. The $k_i^-$ and $k_i^+$ are the number of red squares in the left and right blue shaded squares, counting out from the middle, so $k_1^-=1$, $k_1^+=2$, $k_2^-=5$, and $k_2^+=2$. The $S_i$ are the sum of the lengths of the corresponding blue shaded squares, so $S_1=12$ and $S_2=22$

Let $v_{i}$ denote the permutation of $-[S_{i-1}/2,S_{i}/2]\cup [S_{i-1}/2,S_{i}/2]$ defined by restricting $\widehat{w}^{rev}$ on this interval. For $i=1$, we instead use $-[(S_0-1)/2,S_{1}/2]\cup [(S_0-1)/2,S_{1}/2]$ as $S_0$ is always odd. In Figure \ref{fig: sym process odd}, the $v_{i}$ correspond to the permutations contained in a symmetric pair of blue shaded squares, counting out from the middle.

The utility of this construction is that the cycle structure of $v^{-n,n}$, which we will use to denote the relative order of $\widehat{w}^{rev}(i)$ for $i\in [-n,n]$, can be obtained from the cycle structures of the $v_i$, exactly analogous to how for $q<1$ they could be obtained from the $w_i$.

The following proposition formalizes the fact that the $v_i$ and $S_i$ form a regenerative structure.

\begin{proposition}
\label{prop: sym renewal odd}
The process $(S_i,v_i)$ is regenerative in the following sense. The $(S_i,v_i)$ are independent, and except for $i=0$, the $S_i$ are identically distributed, and the $v_i$, relabeled so they form a permutation of $\{1,\dotsc, S_i\}$, are also identically distributed.
\end{proposition}
\begin{proof}
We proceed by induction, showing each $(S_i,v_i)$ is independent of all previous $(S_j,v_j)$ for $j<i$. The base case is trivial. Given that the $(S_i,v_i)$ are independent for $i\leq k$, we note that $S_{k+1}$ and $v_{k+1}$ are functions of the renewed process, and so are independent of the $(S_i,v_i)$ for $i\leq k$. Finally, since $S_i$ and $v_i$ are the same functions applied to the renewed process, they have the same distribution.
\end{proof}

\begin{figure}
    \centering
    \includegraphics{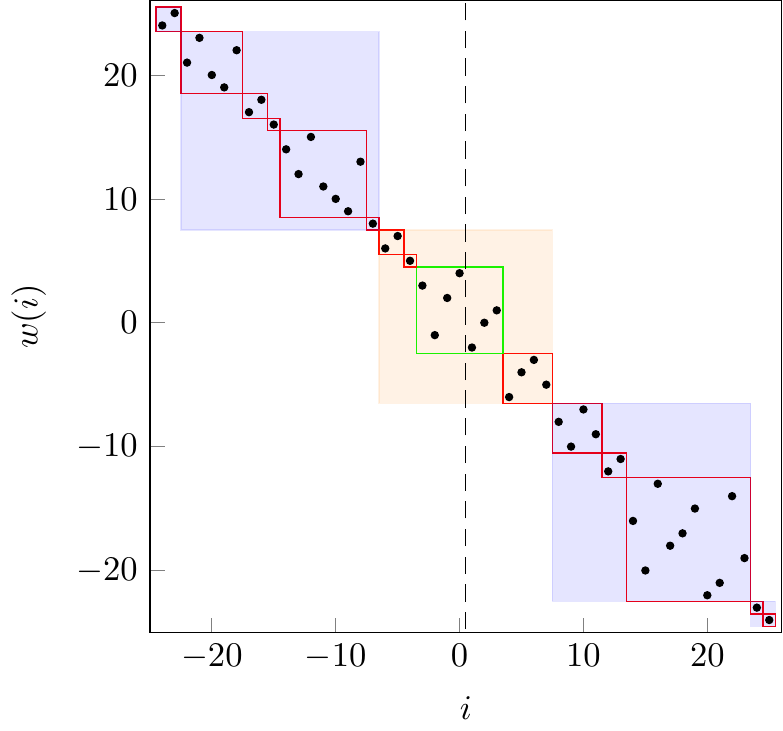}
    \caption{The reversed stationary Mallows process $\widehat{w}^{rev}$ is shown as black dots. Here, we consider the even case, which is centered at $1/2$, marked by the dashed line. Red squares mark the excursions in the stationary Mallows process $\widehat{w}$, with the green square denoting the (possibly empty) excursion containing $1/2$. The orange shaded square denotes the first regeneration in the symmetric process, $S_0$ is the length of the orange shaded square, $v_0$ is the permutation contained in the orange shaded square. The blue shaded squares, paired symmetrically, denote the subsequent regenerations in the symmetric process, $S_i$ is the sum of the lengths of a symmetric pair of blue shaded squares, and $v_i$ is the permutation contained in a symmetric pair of blue shaded squares.}
    \label{fig: sym process even}
\end{figure}

When $n$ is even, we slightly modify the construction above. Figure \ref{fig: sym process even} gives a visual guide to the construction, which is almost identical except that the line of symmetry is at $1/2$ rather than $0$. Since the construction is extremely similar, we only explain the differences in detail.

We wait until the first time that $\widehat{w}^{rev}$ sends $[-n+1,n]$ to itself, and we let $v_0$ be the permutation of this interval that this induces. However, if $\widehat{w}^{rev}$ has a renewal at $1$ (in other words, if $\widehat{w}$ had a renewal at $0$), we simply take $v_0$ to be the empty permutation. Let $S_0$ denote the size of $v_0$. Essentially, the difference is that we center the reversal around $1/2$ rather than $0$. See Figure \ref{fig: sym process even}, where the dashed line is now at $1/2$ rather than at $0$.

We then define $S_i$ in exactly the same way, except we consider when $\widehat{w}^{rev}$ sends the interval $[a,b]$ to $-[a-1,b-1]$ rather than $-[a,b]$, again to account for the fact that everything is centered around $1/2$. The following proposition states that this process also has a regenerative structure, and has the same proof as Proposition \ref{prop: sym renewal odd}. We will use the same notation for both these constructions, since they are essentially the same. In fact, the distributions of the $S_i$ and $v_i$ are identical, except for when $i=0$, since they depend on the process after a renewal.
\begin{proposition}
	\label{prop: sym renewal even}
	The process $(S_i,v_i)$ is regenerative in the following sense. The $(S_i,v_i)$ are independent, and except for $i=0$, the $S_i$ are identically distributed, and the $v_i$, relabeled so they form a permutation of $\{1,\dotsc, S_i\}$, are also identically distributed.
\end{proposition}

\section{Moment estimates for the symmetric process}
\label{sec: moment estimates}
To apply the necessary renewal theory arguments, we need to obtain moment estimates for the $S_i$. We first recall a Markovian representation of the $T_i$ used in \cite{BB17} to prove that they have finite first and second moments, and we refer the reader there for proofs of any unproven claims in this section.

The regeneration times $T_i$ can be viewed as the return times of a Markov chain. Specifically, consider the process
	\begin{equation*}
	M_n=\max_{1\leq i\leq n}w(i)-n
	\end{equation*}
	on $\mathbf{N}$. This is a positive recurrent Markov process with stationary distribution
	\begin{equation*}
	\mu_j=\frac{1}{\prod_{k=1}^\infty (1-q^k)}\frac{q^j}{\prod _{k=1}^j(1-q^k)}.
	\end{equation*}
	The Markov process can be described in terms of the geometric random variables defining the Mallows process. Specifically, the walk can be described as moving from $M_n$ to $M_{n+1}=\max(M_n,Z_{n})-1$ where the $Z_n$ are independent geometric random variables. Let $R_i$ denote the hitting time of $i$ and let $R_i^+$ denote the return time at $i$. Then if the chain is started from $0$, $R_0^+$ is distributed as the size of an excursion in the Mallows process. This Markov chain was used in \cite{BB17} to show that the $T_i$ have finite first and second moments, and will also be essential to show that the $S_i$ have finite moments. A related Markov chain was considered in \cite{GP18}.
	
We now consider the regeneration times $S_i$. The regeneration times $S_i$ ask when two independent copies of the Markov chain starting from $(0,0)$ return to $(0,0)$ simultaneously, since both sides of the stationary Mallows process behave like the one-sided Mallows process. When $n$ is odd, for $S_0$, we can write it as $T_0+S_0'$, where $T_0$ is the length of the excursion containing $0$ and $S_0'$ asks when two independent copies of the Markov chain, started from a state depending on when the excursion $w_0$ starts and ends, hits $(0,0)$, and similarly for $n$ even except we consider the excursion containing $1/2$ (meaning it either contains both $0$ and $1$, or is defined to be empty if this isn't possible).

More formally, consider the Markov chain $(M_n,M_n')$ consisting of two independent copies of the above chain. We let $R_{(i,j)}$ denote the hitting time of $(i,j)$ and $R_{(i,j)}^+$ denote the return time at $(i,j)$. The next proposition shows that the $S_i$ are essentially given by these return times.

\begin{proposition}
\label{prop: S_i=R}
For $i>0$, $S_i$ is equal to $2R_{(0,0)}^+$ in distribution, and $S_0$ is equal in distribution to $2R_{(0,0)}$ where the Markov chain is started at some random location.
\end{proposition}
\begin{proof}
First, suppose that $i>0$. Since $-\widehat{w}(-i)$ is equal in distribution to $\widehat{w}(i)$, the stationary Mallows process to the left of a renewal is equal in distribution to the usual one-sided Mallows process after a shift and reflections across the $x$ and $y$ axes. Moreover, it's independent of the process to the right of a renewal that occurs after the first one. Thus, $S_i$ simply asks for the first time that $\sum T_i=\sum T_i'$, where $T_i$ and $T_i'$ are the regeneration times in two independent copies of the one-sided Mallows process. This is exactly $R_{(0,0)}$, except we multiply by a factor of $2$ due to the definition of $S_i$.

Now take $i=0$, and first, suppose $n$ is odd. When $i=0$, we can write $S_0=T_0+S_0'$, where $T_0$ is the length of the excursion $w_0$ containing $0$. But then $S_0'$ is simply $2R_{(0,0)}$ for the Markov chain started at a random point given by taking the shorter side of the excursion $w_0$, and running the component corresponding to that side until it reaches the same length. 

When $n$ is even, the same proof applies, except that we consider $w_0$ the excursion that contains both $0$ and $1$, or is empty if $0$ and $1$ are contained in distinct excursions.
\end{proof}

We also define $B_i$ to be the set of points $(j,k)$ where $\max(j,k)=i$, and we let $R_{B_i}$ denote the hitting time of $B_i$. The proof that these times have finite moments then proceeds similarly to the analogous proof for the $T_i$, given in \cite{BB17, H21}.

We will let $\E_x$, $\P_x$, and $\E_{\mu\times \mu}$, $\P_{\mu\times\mu}$ denote the expectation and probability measure for the Markov chain started from $x\in\mathbf{N}^2$ and from stationarity respectively.

\begin{lemma}
\label{lem: moment ineq}
For any point $x\in B_{i+1}$ and any point $y\in B_i$, we have $\E_x(R_{B_i}^k)\leq \E_y(R_{B_{i-1}}^k)$.
\end{lemma}
\begin{proof}
We couple two copies of the chain, started from $x$ and $y$, using the same underlying geometric random variables. Recall that we sample a step of the Markov chain by sampling geometric random variables $Z$ and $Z'$, and setting $(M_{n+1},M_{n+1}')=(\max(M_n,Z)-1,\max(M_n',Z')-1)$. Thus, each component can decrease by at most $1$ in a single step. The sets $B_i$ are constructed such that the Markov chain, if it starts on $B_{k}$ for any $k>i$, must hit $B_i$ before $B_{i-1}$. Thus, if the two copies ever collide before hitting $B_{i-1}$, they must both hit $B_i$ before $B_{i-1}$.

Consider the next step, generated by geometric random variables $Z$ and $Z'$. First, note that if both $Z\geq i+1$ and $Z'\geq i+1$, then both chains will couple, and so $B_i$ must be hit by both chains before $B_{i-1}$ can be hit.

If $Z\geq i+1$ but $Z'\leq i$, then the two copies of the chain couple in the first coordinate, and the second coordinates are both at most $i$. Then either in the subsequent steps, the $Z'$ are all at most $i$, in which case the two chains will eventually hit $B_i$ before $B_{i-1}$, or the $Z'$ will eventually be at least $i+1$, in which case both chains will couple. The case when $Z\leq i$ and $Z'\geq i+1$ is similar.

On the other hand, if both $Z\leq i$ and $Z'\leq i$, then the two chains will hit $B_i$ and $B_{i-1}$ respectively in the first step, and so the hitting times are equal. But then we've shown that $R_{B_i}\leq R_{B_{i-1}}$ under this coupling, and the result follows.
\end{proof}

\begin{lemma}
\label{lem: finite return time}
For all $k$, $\E_{(0,0)}((R_{(0,0)}^+)^k)<\infty$.
\end{lemma}
\begin{proof}
We proceed by induction on $k$. The case $k=1$ follows from the existence of the stationary distribution. We now assume that $\E_{(0,0)}((R_{(0,0)}^+)^k)<\infty$. We first show that $\E_{\mu\times \mu}(R_{(0,0)}^k)<\infty$.

To see this, note that
\begin{equation*}
    \E_{\mu\times\mu}(R_{(0,0)}^{k})=\sum_{i,j}\mu_i\mu_j\left(\sum _{l=1}^{\max(i,j)} T_{l\to l-1}\right)^k\leq \sum_{i,j}\mu_i\mu_j \max(i,j)^k \max_{x\in B_1}\E_{x}(R_{(0,0)}^k)^k,
\end{equation*}
where $T_{l\to l-1}$ denotes the time it takes to reach $B_{l-1}$ starting from the first time we hit $B_{l}$, and we expand into $\max(i,j)^k$ terms and use Lemma \ref{lem: moment ineq} to replace $\E(T_{l\to l-1}^l)$ with $\max_{x\in B_1}\E_{x}(R_{(0,0)}^k)$, which will dominate $T_{1\to 0}$ no matter the distribution of the initial state. But since $\mu_i\mu_j\leq Aq^{i+j}$ for some $A>0$, and since $\E_{x}(R_{(0,0)}^k)<\infty$ by the inductive hypothesis (since we can always move from $(0,0)$ to $x$ in the first step), we've shown $\E_{\mu\times\mu}(R_{(0,0)}^{k})$ is finite.

Then Lemma 2.23 of \cite{AF} states
		\begin{equation*}
		\P_{\mu\times \mu}(R_{(0,0)}=t-1)=\mu_0 ^2\P_{(0,0)}(R_{(0,0)}^+\geq t)
		\end{equation*}
		which gives
		\begin{equation*}
		\E_{(0,0)}(P_{k+1}(R_{(0,0)}^+))=\frac{\E_{\mu\times \mu}((R_{(0,0)}+1)^{k})}{\mu_0^2}
		\end{equation*}
		after multiplying by $t^{k}$ and summing over $t$, where
		\begin{equation*}
		P_{k+1}(m)=\sum _{i=1}^m i^{k}
		\end{equation*}
		are the Faulhaber polynomials of degree $k+1$. This implies that $\E_{(0,0)}((R_{(0,0)}^+)^{k+1})<\infty$ since all other expressions in the equation are finite by the inductive hypothesis.
\end{proof}

\begin{proposition}
\label{prop: S_i estimates}
For both $n$ even and $n$ odd, we have $\E(S_i^k)<\infty$ for all $i$ and $k$.
\end{proposition}
\begin{proof}
When $i>0$, this is immediate by Lemma \ref{lem: finite return time} since $S_i$ is equal in distribution to $2R_{(0,0)}^+$ by Proposition \ref{prop: S_i=R}. 

For $i=0$, we note that $S_0$ is equal in distribution to $2R_{(0,0)}^+$, but started from a random state given by running one component of the Markov chain for a random time $T$ and conditioning on the second component being $0$. We can bound the return time to $0$ starting from the random state given by running the chain until time $T$ by $R_1+\dotsc+R_{T+1}$, where $R_i$ is $i$th return time to $0$. The $k$th moment of this is bounded by $\E((T+1)^k)\E(R_1^k)^k$ which is finite since $\E(R_1^k)\leq \E((R_{(0,0)}^+)^k)$ and $T$ is bounded by the length of the excursion containing $0$, which has the size-bias distribution of the excursion lengths and so has finite moments. Finally, since the two components are independent, we can condition on the second component being $0$, which has positive probability.
\end{proof}

\section{Proofs of main theorems}
\label{sec: main}

\subsection{Proofs of Theorems \ref{thm: q<1} and \ref{thm: q>1}}
With the regenerative structure given by the Mallows process and the moment estimates given by Proposition \ref{prop: T_i estimates}, Theorem \ref{thm: q<1} follows immediately from Theorems \ref{thm: regen CLT} and \ref{thm: moment conv}.

\begin{proof}[Proof of Theorem \ref{thm: q<1}]
We note that $C_i(w)$, the number of cycles, is additive in the sense that $C_i(w)=C_i(w_1)+C_i(w_2)$ if $w=w_1w_2$ and $w_1,w_2$ send disjoint intervals to themselves. Thus, if we identify the random Mallows permutation $w\in S_n$ with the relative order of $\widehat{w}(i)$ for $i\leq n$, we obtain
\begin{equation*}
    C_i(w)=\sum_{j=1}^{N(n)}C_i(w_j)+C_i(w'),
\end{equation*}
where $w_j$ are the excursions in the Mallows process $\widehat{w}$ that end before time $n$, and $w'$ is the relative order of $\widehat{w}(j)$ for $j$ after the end of the last excursion.

Now we note that $\frac{C_i(w')}{\sqrt{n}}\to 0$ almost surely, as $C_i(w')\leq T^{(n)}\xrightarrow{(d)} T^*$ by Proposition \ref{prop: size-bias}, where $T^{(n)}$ is the size of the excursion containing $n$ and $T^*$ has the size-bias distribution with respect to the $T_i$. Since $C_i(w_j)\leq T_j$, all moments exist for the $C_i(w_j)$ by Proposition \ref{prop: T_i estimates}. Thus, Theorem \ref{thm: regen CLT} gives the joint convergence in distribution for any finite collection of the $C_i$, and thus also the convergence in distribution for the entire vector.

The mean and covariance asymptotics follow from Theorem \ref{thm: moment conv}. We note that $\beta_{ii}>0$ for all $i$, because we can condition on $T_1=k$ for some large $k$ (which occurs with positive probability), and $C_i$ will not be constant on the irreducible permutations in $S_k$. Finally, the statement about $C(w)$ follows from the exact same argument, applied to $C$ rather than $C_i$.
\end{proof}

The proof of Theorem \ref{thm: q>1} is similar to that of Theorem \ref{thm: q<1}, except that we work with the symmetric process rather than the Mallows process.

\begin{proof}[Proof of Theorem \ref{thm: q>1}]
We note that $C_{i}(w)$, is anti-additive in the sense that $C_i(w)=C_i(w_1)+C_i(w_2)$ if $w=w_1w_2$ and for some $k$, $w_1$ sends $[k+1,n-k]$ to itself while $w_2$ sends $[1,k]$ to $[n-k+1,n]$ and vice versa. 

We first consider the case when $n$ is odd. If we identify the random Mallows permutation $w\in S_n$ with the relative order of $\widehat{w}^{rev}(k)$ for $k\in [-(n-1)/2,(n-1)/2]$, we obtain
\begin{equation*}
    C_{2i}(w)=\sum_{j=0}^{N(n)}C_{2i}(v_j)+C_{2i}(v'),
\end{equation*}
where $N(n)$ is the largest $t$ for which $\sum_{j\leq t}S_j\leq n$, $v_j$ are the excursions in the symmetric process that end before time $n$, and $v'$ is the relative order of $\widehat{w}^{rev}(k)$ for $|k|>\sum_{j=0}^{N(n)}S_j/2$. When $n$ is even, we do essentially the same thing, except we will work on the interval $[-n/2+1,n/2]$ instead.

But now, we note that $\frac{C_{2i}(v')}{\sqrt{n}}\to 0$ almost surely as $C_{2i}(v')\leq S^{(n)}\xrightarrow{(d)} S^*$ by Proposition \ref{prop: size-bias} (where $S^{(n)}$ is the size of the $v_i$ containing $n/2$ or $(n-1)/2$, and $S^*$ has the size-bias distribution with respect to the $S_i$), and $C_{2i}(v_j)\leq S_j$ so all moments exist for the $C_{2i}(v_j)$ by Proposition \ref{prop: S_i estimates}, and so Theorem \ref{thm: regen CLT} gives the joint convergence in distribution for any finite collection of the $C_{2i}(w)$, and thus also the convergence in distribution for the entire vector.

The mean and covariance asymptotics follow from Theorem \ref{thm: moment conv}. Again, $\beta_{ii}>0$ because $C_{2i}$ will not be constant on the permutations in $S_{2k}$ sending $[1,k]$ to $[k+1,2k]$ and vice versa for large enough $k$. The statement on $C(w)$ follows from the exact same arguments applied to $C$ rather than $C_{2i}$.

Finally, since for all odd $i$, $C_i(v_j)=0$ for all $j>0$ and $C_i(v')=0$, we have $C_i(w)=C_i(v_0)$ as long as $v_0$ lies in the interval $[-(n-1)/2,(n-1)/2]$ (or $[-n/2+1,n/2]$ if $n$ is even). But the probability of this event goes to $1$ as $n\to\infty$, and so the $C_i(w)$ will converge almost surely to $C_i(v_0)$ for $i$ odd. Then $w^{odd}$ and $w^{even}$ can be taken to be $v_0$ under the odd and even constructions, which were used depending on whether $n$ was odd or even.
\end{proof}

\begin{remark}
\label{rmk: constants}
From the proof, it's clear that the constants $\alpha_i$, $\beta_{ij}$, and $\beta$, and similarly the $\alpha_i'$, $\beta_{ij}'$, and $\beta'$, are simply what is given by Theorem \ref{thm: moment conv} applied to either the Mallows process or the symmetric process. For convenience, we give the explicit expressions. Let $\mu=\E(T_1)$ and $\mu'=\E(S_1)$. Then we have
\begin{align*}
    \alpha_i&=\frac{1}{\mu}\E(C_i(w_1)),
    \\\beta_{ij}&=\frac{1}{\mu}\Cov\left(C_i(w_1)-\frac{\E(C_i(w_1))T_1}{\mu},C_j(w_1)-\frac{\E(C_j(w_1))T_1}{\mu}\right),
    \\\beta&=\frac{1}{\mu}\Var\left(C(w_1)-\frac{\E(C(w_1))T_1}{\mu}\right),
\end{align*}
and
\begin{align*}
        \alpha_i'&=\frac{1}{\mu'}\E(C_{2i}(v_1)),
        \\\beta_{ij}'&=\frac{1}{\mu'}\Cov\left(C_{2i}(v_1)-\frac{\E(C_{2i}(v_1))S_1}{\mu'},C_{2j}(v_1)-\frac{\E(C_{2j}(v_1))S_1}{\mu'}\right),
        \\\beta'&=\frac{1}{\mu'}\Var\left(C(v_1)-\frac{\E(C(v_1))S_1}{\mu'}\right).
\end{align*}
\end{remark}

\begin{remark}
\label{rmk: proof of ext}
Examining the proofs of Theorem \ref{thm: q<1}, we see that the only properties of the functions $C_i$ we use are that they are additive, in the sense that $C_i(w)=\sum C_i(w_j)+C_i(w')$, that we can control the moments of $C_i(w_j)$, and that the variances $\Var\left(C_i(w_1)-\frac{\E(C_i(w_1))T_1}{\mu}\right)$ are non-zero.

If $f$ is assumed to be additive, then the first property holds. If for $w\in S_n$, $|f(w)|\leq Cn^k$ for constants $C$ and $k$, then $|f(w_i)|\leq CT_i^k$ and since all moments exist for the $T_i$, the second property holds. To ensure that the variance is non-zero, it suffices to assume that there exists some $k$ such that conditional on $T_1=k$, $f(w_1)$ is non-constant, and this is what we mean be a non-trivial function. Essentially, there must be some $k$ for which $f$ is not constant on the irreducible permutations in $S_k$. With these assumptions, an analogue of Theorem \ref{thm: q<1} then holds with the exact same proof.

A similar explanation explains why anti-additivity and the bound $|f(w)|\leq Cn^k$ for $w\in S_n$ along with the assumption that there is some $k$ for which $f(v_1)$ conditional on $S_1=k$ is non-constant is enough to show an analogue of Theorem \ref{thm: q>1}.
\end{remark}

\subsection{Aymptotics for fixed points}
\label{sec: fixed points}
We can compute the asymptotics for the number of fixed points when $q<1$ by computing $\alpha_1$.

\begin{proposition}
\label{prop: a1}
Let $q<1$ and let $w\in S_n$ be Mallows distributed. We have
\begin{equation*}
    \alpha_1=\lim_{n\to\infty}n^{-1}\E(C_1(w))=\frac{1-q}{q}(q;q)_\infty\sum_{j=0}^\infty \frac{q^{(j+1)^2}}{(q;q)_j^2}.
\end{equation*}
\end{proposition}
\begin{proof}
Since $w\in S_n$ and $w^{1,n}$ agree on all but the last excursion, which has finite mean, we have that $\alpha_1$ can be computed as the probability that $\widehat{w}(i)=i$ in the stationary Mallows process. The formula then follows from a special case of Theorem 5.1 in \cite{GO12}, when in their notation, $d=0$.
\end{proof}
\begin{remark}
Proposition \ref{prop: a1} can also be proven directly from the finite model, without using the Mallows process. This involves computing $\P(w(i)=i)$ for all $i$, which has a messy but explicit formula, and then computing $n^{-1}\sum_i \P(w(i)=i)$, which actually simplifies the sums somewhat.

In fact, one can obtain formulas for all the $\alpha_i$, but the expressions quickly become much more complicated, and do not appear useful for further analysis. We thus do not attempt to write down formulas for the other $\alpha_i$.
\end{remark}

\section*{Acknowledgment}
The author wishes to thank Persi Diaconis for bringing this problem to his attention. 

\bibliography{bibliography}{}
\bibliographystyle{habbrv.bst}

\end{document}